\documentclass[12pt,a4paper]{amsart}
\usepackage[utf8]{inputenc}
\usepackage{enumerate,bbm,titletoc,a4,mathtools,mathrsfs,amssymb,amsthm}
\usepackage[usenames,dvipsnames]{xcolor}
\usepackage[pagebackref,colorlinks,linkcolor=BrickRed,citecolor=OliveGreen,urlcolor=black,hypertexnames=true]{hyperref}

\DeclareMathOperator{\rk}{rk}
\DeclareMathOperator{\GL}{GL}
\DeclareMathOperator{\SL}{SL}

\DeclareMathOperator{\tr}{tr}
\DeclareMathOperator{\opm}{M}

\newcommand{\de}{\delta}

\newcommand{\N}{\mathbb{N}}

\newcommand{\C}{\mathbb{C}}

\newcommand{\cD}{\mathcal{D}}

\newcommand{\cO}{\mathcal{O}}

\newcommand{\cZ}{\mathcal{Z}}
\newcommand{\kk}{\mathbbm k}
\newcommand{\ti}{{\rm t}}

\newcommand{\uX}{{\underline X}}
\newcommand{\uY}{{\underline Y}}
\newcommand{\ux}{{\underline x}}

\newcommand*{\mtx}[1]{\opm_{#1}(\kk)}

\newcommand{\Langle}{\mathop{<}\!}
\newcommand{\Rangle}{\!\mathop{>}}

\newcommand{\px}{\kk\!\Langle \ux\Rangle}

\newtheorem{thm}{Theorem}[section]
\newtheorem{prop}[thm]{Proposition}

\newtheorem{conj}[thm]{Conjecture}
\theoremstyle{definition}

\theoremstyle{remark}

\title[Dimension-free Nullstellens\"atze for nc polynomials]{Dimension-free matricial Nullstellens\"atze for noncommutative polynomials}
\author{Jurij Vol\v{c}i\v{c}}
\address{Department of Mathematics, Drexel University, Pennsylvania}
\email{jurij.volcic@drexel.edu}
\date{\today}



\begin{document}
\maketitle

\begin{abstract}
Hilbert's Nullstellensatz is one of the most fundamental correspondences between algebra and geometry, and has inspired a plethora of noncommutative analogs.
In last two decades, there has been an increased interest in understanding vanishing sets of polynomials in several matrix variables without restricting the matrix size, prompted by developments in noncommutative function theory, control systems, operator algebras, and quantum information theory.
The emerging results vary according to the interpretation of what vanishing means. For example, given a collection of noncommutative polynomials, one can consider all matrix tuples at which the values of these polynomials are all zero, singular, have common kernel, or have zero trace. 
This survey reviews Nullstellens\"atze for the above types of vanishing sets, and identifies their structural counterparts in the free algebra.
\end{abstract}

\section{Introduction}

The classical Hilbert's Nullstellensatz \cite{Hil} is a cornerstone of algebraic geometry. The successful interplay between commutative algebra and geometry inspired analogous advances in noncommutative algebraic geometry \cite{VV,Ros,BRSSW}. As opposed to the classical commutative theory, noncommutative variants of algebraic geometry are rather diverse, both in scope and goals; in particular, there are several results that can be interpreted as noncommutative Nullstellens\"atze \cite{Ami,Coh0,Irv,BR}. This survey focuses on a distinct family of such theorems, where vanishing sets of polynomials on matrix tuples of arbitrary sizes play the role of algebraic sets.

Let $\kk$ be an algebraically closed field of characteristic 0, and let $\ux=(x_1,\dots,x_d)$ be a tuple of freely noncommuting variables. Let $\px$ be the free $\kk$-algebra generated by $x_1,\dots,x_d$, equipped with the standard degree function $\deg$. Elements of $\px$ are called \emph{noncommutative polynomials}. 
In this paper we are interested in evaluations of noncommutative polynomials on points in $\mtx{n}^d$ for every $n\in\N$, and their vanishing features. That is, every $f\in\px$ and $\uX=(X_1,\dots,X_d)\in\mtx{n}^d$ give rise to $f(\uX)\in\mtx{n}$ in a natural way.

Before looking at the general motivation behind this perspective, let us comment on why it is reasonable to view noncommutative polynomials as functions in dimension-independent matrix variables. For a general mathematician, matrix algebras are the most familiar noncommutative rings; moreover, noncommutative polynomials are determined by their evaluations on matrices of all sizes (in other words, the finite-dimensional representations distinguish elements of the free algebra), but not by their evaluations on matrices of a fixed size (because of the existence of polynomial identities \cite{Row}).

The above perspective on noncommutative polynomials as functions in several matrix variables of arbitrary sizes became especially prominent with the rise of free analysis \cite{KVV,AM} and free real algebraic geometry \cite{dOHMP,HKM}. The former develops the analytic theory of noncommutative polynomials and more general noncommutative functions, while the latter investigates noncommutative inequalities, their geometry and optimization. These topics have been fueled by dimension-independent problems in control theory, free probability, semidefinite optimization, and quantum information theory. Several algebraic results results arising from free analysis and free real algebraic geometry can be seen as noncommutative Nullstellens\"atze, connecting vanishing of noncommutative polynomials with ring-theoretic concepts in the free algebra. 

Throughout the text, let $f_1,\dots,f_\ell,g\in\px$ be noncommutative polynomials. Suppose $g$ vanishes at all matrix tuples where all $f_1,\dots,f_\ell$ vanish; what can be said about the structural relation between $f_1,\dots,f_\ell$ and $g$ in the free algebra? The answer to this question depends on the precise definition of ``vanishing''; namely, one might be interested in matrix tuples where the values of polynomials are zero, or singular, or have zero trace, and so on.
This survey compares various different vanishing scenarios, 
reviews corresponding Nullstellens\"atze that relate them to their counterparts in the free algebra (e.g., ideals, left ideals, factorization), and poses some open questions about the geometry-algebra correspondences in the context of dimension-free matricial evaluations of noncommutative polynomials.
It is also worth acknowledging what this survey does not explore. Closely related theorems consider the free algebra with involution and its $*$-representations, leading to so-called real Nullstellens\"atze \cite{PS,Scm,CHMN,KSV,HKV1}. These results often have a more functional-analytic flavor; in this text, we restrict to the more ring-theoretic statements.

\section{True zeros}\label{sec2}

This section addresses the most standard zero sets of noncommutative polynomials. For $f_1,\dots,f_\ell\in\px$ let
$$\cZ(f_1,\dots,f_\ell) = \bigcup_{n\in\N}
\left\{
\uX\in\mtx{n}^d\colon f_j(\uX)=0 \ \forall j
\right\},
$$
and let $(f_1,\dots,f_\ell)$ be the two-sided ideal in $\px$ generated by $f_1,\dots,f_\ell$.
One can view $\cZ(f_1,\dots,f_\ell)$ as the collection of all finite-dimensional representations of the $\kk$-algebra $\px/(f_1,\dots,f_\ell)$.
In particular, this observation immediately suggests that zero sets of noncommutative polynomials share a close connection with ideals in $\px$.

For every $n\in\N$ let $\operatorname{PI}_n$ denote the ideal in $\px$ of all polynomials that are constantly zero on $\mtx{n}^d$ (i.e., $d$-variate polynomial identities for $n\times n$ matrices).
The first result about vanishing on $\cZ(f_1,\dots,f_\ell)$ is a consequence of Amitsur's Nullstellensatz \cite{Ami}.

\begin{thm}[Amitsur]\label{t:ami}
If $\cZ(f_1,\dots,f_\ell)\subseteq\cZ(g)$ then for every $n\in\N$ there exists $r\in\N$ such that $g^r\in (f_1,\dots,f_\ell)+\operatorname{PI}_n$.
\end{thm}

Theorem \ref{t:ami} is an immediate consequence of \cite[Theorem 1]{Ami}, which employs the results on central simple algebras and polynomial identities. Let us comment on the two aspects of Theorem \ref{t:ami}. Firstly, its statement is not an equivalence. For example, let $f_1=x_1^2$ and $g=x_1$. Then $g^2\in (f_1)$ but $\cZ(f_1)\not\subseteq \cZ(g)$ (due to 2-nilpotent matrices). Secondly, Theorem \ref{t:ami} is not fully in the spirit of this survey, in the sense that its algebraic part still requires a quantifier referring to matrix sizes. Namely, one might optimistically hope for a statement of the form ``$\cZ(f_1,\dots,f_\ell)\subseteq\cZ(g)$ if and only if $g\in(f_1,\dots,f_\ell)$'' to be true. However, this assertion is false in general. A fundamental obstacle is that a (nonconstant) noncommutative polynomial might simply not have any matrix zeros. For example, let $f_1=1-[x_1,x_2]$ and $g=1$ (here, $[a,b]=ab-ba$ is the additive commutator). Then $\cZ(f_1)=\emptyset$ (note that $\tr(I-[X_1,X_2])=\tr(I)\neq0$ since $\operatorname{char}\kk=0$). In particular, $\cZ(f_1)\subseteq \cZ(g)$ but $1\notin (f_1)$. Section \ref{sec6} below speculates about a possible way around this issue. 

Nevertheless, the aforementioned optimistic assertion is valid for certain families of constraints $f_1,\dots,f_\ell$; let us present two of them.
The next result strengthens Theorem \ref{t:ami} under the assumption that $f_1,\dots,f_\ell$ are \emph{homogeneous} with respect to the natural grading of $\px$.

\begin{thm}[Salomon-Shalit-Shamovich]\label{t:SSS}
Suppose $f_1,\dots,f_\ell$ be homogeneous.
Then $\cZ(f_1,\dots,f_\ell)\subseteq\cZ(g)$ if and only if $g\in (f_1,\dots,f_\ell)$.
\end{thm}

See \cite[Theorem 7.3]{SSS} for the original statement. Its language and proof have a somewhat analytic style; let us rephrase the argument for the nontrivial implication in Theorem \ref{t:SSS}, and slightly strengthen it. Suppose $g\notin (f_1,\dots,f_\ell)$, and let $J$ denote the ideal in $\px$ generated by products of variables in $\ux$ of length $\deg g+1$. By homogeneity of $f_j$, it follows that $g\notin (f_1,\dots,f_\ell)+J$. Therefore the image of $g$ under the canonical quotient homomorphism
$$\px\to \px\big/\big((f_1,\dots,f_\ell)+J\big)$$
is nonzero. Note that the unital algebra $\px\big/\big((f_1,\dots,f_\ell)+J\big)$ is finite-dimensional and generated by nilpotents. More precisely, by considering its left-regular representation, we can view it as a unital subalgebra of $\mtx{N}$ with $N=\sum_{i=0}^{\deg g}d^i=\frac{d^{\deg g+1}-1}{d-1}$, generated by jointly nilpotent matrices $X_1,\dots,X_d$ of nilpotency order at most $\deg g+1$. That is, every product of $\deg g+1$ factors from $\{X_1,\dots,X_d\}$ equals zero. Then $f_j(\uX)=0$ for all $j$, and $g(\uX)\neq0$.

The conclusion of Theorem \ref{t:SSS} also holds for certain non-homogeneous ideals, which are called \emph{rationally resolvable} \cite{KVV1}. The notion of a rationally resolvable ideal is somewhat technical; roughly speaking, an ideal $J$ in $\px$ is rationally resolvable if one can split $\ux$ into two groups, and variables in the second group can be expressed as rational expressions in variables of the first group modulo $J$. Without stating an exact definition, let us give three examples of rationally resolvable ideals.
\begin{enumerate}
    \item Let $d\ge 3$, and let $f\in\px$ be a nonconstant polynomial in variables $x_1,\dots,x_{d-2}$ only. Then the ideal $(x_{d-1}x_d-f)\subset\px$ is rationally resolvable \cite[Lemma 4.3]{KVV1}.
    \item Let $w_1,w_2,\neq1$ be two products of variables in $\ux$ such that $w_1$ and $w_2$ do not start or end with the same variable, and a distinguished variable appears exactly once in $w_1$ and does not appear in $w_2$. Then $(w_1-w_2)\in\px$ is rationally resolvable \cite[Example 2.7]{KVV1}. 
    \item Let $d=2m^2$ for some $m\in\N$, and denote the variables of $\ux$ as $x_{ij},y_{ij}$ for $1\le i,j\le m$.
    Consider the $m\times m$ matrices 
    $\mathfrak{X}=(x_{ij})_{i,j}$ and $\mathfrak{Y}=(y_{ij})_{i,j}$. Then the ideal in $\px$ generated by the entries of $\mathfrak{X}\mathfrak{Y}-I_m$ and $\mathfrak{Y}\mathfrak{X}-I_m$ is rationally resolvable.
\end{enumerate}

For rationally resolvable constraints $f_1,\dots,f_\ell$, the conclusion of Theorem \ref{t:SSS} holds without the homogeneity assumption.

\begin{thm}[Klep-Vinnikov-Vol\v{c}i\v{c}]\label{t:KVV1}
Suppose $f_1,\dots,f_\ell\in\px$ generate a rationally resolvable ideal (e.g. one the aforementioned). Then $\cZ(f_1,\dots,f_\ell)\subseteq\cZ(g)$ if and only if $g\in (f_1,\dots,f_\ell)$.
\end{thm}

See \cite[Theorem 2.5 and Proposition 2.6]{KVV1} for the proof, which relies on the existence of the universal skew field of fractions of $\px$, the so-called \emph{noncommutative rational functions}, and their matrix evaluations.

\section{Directional zeros}\label{sec3}

In this section we consider a one-sided type of zeros.
Given $f_1,\dots,f_\ell\in\px$ let
$$\cZ^{\operatorname{dir}}(f_1,\dots,f_\ell) = \bigcup_{n\in\N}
\left\{
(\uX,v)\in\mtx{n}^d\times \kk^n\colon f_j(\uX)v=0 \ \forall j
\right\},
$$
and let $\mathscr{L}(f_1,\dots,f_\ell)$ denote the left one-sided ideal in $\px$ generated by $f_1,\dots,f_\ell$.
While the zero set $\cZ(f_1,\dots,f_\ell)$ from Section \ref{sec2} has a clear ring-theoretic aspect (in terms of finite-dimensional representations of a quotient ring), 
the set $\cZ^{\operatorname{dir}}(f_1,\dots,f_\ell)$, often referred to as \emph{directional zeros} of $f_1,\dots,f_\ell$,  pertains for example to positivity on analytic noncommutative varieties \cite{HMP} and inner-outer factorization of noncommutative functions \cite{JMS}.
The exact correspondence between directional zeros and left ideals is captured by Bergman's Nullstellensatz \cite{HM}.

\begin{thm}[Bergman]\label{t:berg}
$\cZ^{\operatorname{dir}}(f_1,\dots,f_\ell)\subseteq\cZ^{\operatorname{dir}}(g)$ if and only if $g\in \mathscr{L}(f_1,\dots,f_\ell)$.
\end{thm}

See \cite[Theorem 6.3]{HM} or \cite[Theorem 3]{HMP} for the proof, following a technique that is often referred to as the truncated Gelfand-Naimark-Segal construction. Let us outline this argument, to observe similarities with the proof of Theorem \ref{t:SSS}. Only the forward implication in Theorem \ref{t:berg} is nontrivial. Suppose $g\notin \mathscr{L}(f_1,\dots,f_\ell)$. For $\de=\max\{\deg g,\deg f_1,\dots,\deg f_\ell\}$, 
let $V_\de$ and $V_{\de+1}$ be the images of the subspaces of noncommutative polynomials of degree at most $\de$ and $\de+1$, respectively, under the quotient homomorphism of left $\px$-modules
$$q:\px\to\px\big/ \mathscr{L}(f_1,\dots,f_\ell).$$
Note that $V_\de\subseteq V_{\de+1}$ are finite-dimensional spaces.
Let $\pi:V_{\de+1}\to V_\de$ be any projection, and consider (well-defined) linear maps $X_1,\dots,X_d:V_\de\to V_\de$ given by 
$X_i q(f)=\pi(q(x_if))$ for $f\in\px$ with $\deg f\le \de$. Observe that $f(\uX)q(1)=q(f)$ for $f\in\px$ with $\deg f\le \de$. In particular, $f_j(\uX)q(1)=0$ for all $j$, and $g(\uX)q(1)\neq0$.

Theorem \ref{t:berg} can be seen as an algebraic certificate for inclusion of joint kernels of matricial values of noncommutative polynomials. 
With this in mind, it is also natural to ask if there is an analogous certificate for inclusion of joint invariant subspaces of matricial values of noncommutative polynomials. While such a statement might not be directly considered a Nullstellensatz like the other theorems in this survey (though one could argue that joint invariant subspaces correspond to zeros in a suitable exterior product), it would certainly adhere to the same geometric perspective on noncommutative polynomials.
In particular, one may consider the following assertion.

\begin{conj}\label{c:alg}
For $f_1,\dots,f_\ell,g\in\px$, the following are equivalent:
\begin{enumerate}[(i)]
    \item for all $n\in\N$ and $\uX\in\mtx{n}^d$, joint invariant subspaces for $f_1(\uX),\dots,f_\ell(\uX)$ are invariant for $g(\uX)$;
    \item $g$ belongs to the unital $\kk$-algebra generated by $f_1,\dots,f_\ell$.
\end{enumerate}
\end{conj}

Conjecture \ref{c:alg} is true at least for $\ell=1$, where one can further restrict to 1-dimensional invariant subspaces (i.e., eigenvectors).

\begin{prop} Let $f,g\in\px$. Then $f\in \kk[g]$ if and only if for all $n\in\N$ and $\uX\in\mtx{n}^d$, eigenvectors of $f(\uX)$ are eigenvectors of $g(\uX)$.
\end{prop}

\begin{proof}
Only $(\Leftarrow)$ is nontrivial. Suppose eigenvectors of $f(\uX)$ are eigenvectors of $g(\uX)$, for all matrix tuples $\uX$. If $f$ is constant, then $g$ is likewise constant. Thus we can assume that $f$ is nonconstant. Given $n\in\N$ let $\cO_n\subseteq\mtx{n}^d$ denote the set of all $\uX$ such that $f(\uX)$ has $n$ distinct eigenvalues. By \cite[Corollary 2.10]{BV} (or \cite[Proposition 1.7]{KBRZ}), $\cO_n$ is nonempty for all large enough $n\in\N$. Since $\cO_n$ is Zariski open, it is therefore Zariski dense in $\mtx{n}^d$ for such $n$. By the assumption we have $f(\uX)g(\uX)=g(\uX)f(\uX)$ for all $X\in\cO_n$. Zariski denseness then implies that $f(\uX)g(\uX)=g(\uX)f(\uX)$ for all $\uX\in\mtx{n}^d$ and large enough $n\in\N$. As no nonzero polynomial is constantly zero on matrix tuples of arbitrary large sizes (see e.g. \cite[Lemma 1.4.3]{Row}), it follows that $fg=gf$ in $\px$. By Bergman's centralizer theorem \cite[Theorem 5.3]{Ber} there exist $h\in\px$ and $p_1,p_2\in\kk[t]$ such that $f=p_1(h)$ and $g=p_2(h)$.

For $n\in\N$ let $\Lambda_n\subseteq\kk$ denote the set of eigenvalues attained by $h$ on $\mtx{n}^d$. By Chevalley’s theorem \cite[Theorem 3.16]{Har}, $\Lambda_n$ is a constructible set in $\kk$ (since it is the projection of the algebraic set 
$\{(\lambda,\uX)\in\kk\times\mtx{n}^d\colon \det(\lambda I-h(\uX))=0\}$).
Therefore $\Lambda_n$ is either finite or cofinite in $\kk$. If $\Lambda_n$ is finite, then there is $\lambda\in \Lambda_n$ such that $h(\uX)-\lambda I$ is singular for all $\uX$ in a Zariski dense subset of $\mtx{n}^d$, and consequently $h(\uX)-\lambda I=0$ for all $\uX\in\mtx{n}^d$ by \cite[Theorem 3.2.6]{Row}. Thus if $\Lambda_n$ were finite for all $n\in\N$, the noncommutative polynomial $h$ would be central in $\px$ and therefore constant, which is not the case (as $f$ is nonconstant). Thus let $n\in\N$ be such that $\Lambda_n$ is cofinite in $\kk$. 
Assume $p_2\notin\kk[p_1]$. By \cite[Theorem 2.3]{FM} there exist $\lambda,\mu\in\Lambda_n$ such that $p_1(\lambda)=p_1(\mu)$ and 
$p_2(\lambda)\neq p_2(\mu)$. Let $\uX,\uY\in\mtx{n}^d$ be such that $h(\uX)u=\lambda u$ and $h(\uY)v=\mu v$ for some nonzero $u,v\in\kk^n$.
Then $\left(\begin{smallmatrix}u\\v\end{smallmatrix}\right)\in\kk^{2n}$ is an eigenvector of $(p_1(h))(\uX)\oplus (p_1(h))(\uY)=f(\uX\oplus\uY)$, but not of $(p_2(h))(\uX)\oplus (p_2(h))(\uY)=g(\uX\oplus\uY)$, which is a contradiction. Therefore $p_2\in\kk[p_1]$, and so $g\in\kk[f]$.
\end{proof}

For an alternative geometric counterpart of noncommutative polynomial composition in terms of eigenlevel sets, see \cite[Corollary 4.1]{Vol}.

\section{Determinantal zeros}\label{sec4}

Directional zeros from Section \ref{sec3} capture a lot of structure (which is the reason behind the neat form of Theorem \ref{t:berg}); however, sometimes one would like to be able to draw conclusions based on less information.
Given $f_1,\dots,f_\ell\in\px$ let
$$\cZ^{\operatorname{det}}(f_1,\dots,f_\ell) = \bigcup_{n\in\N}
\left\{
\uX\in\mtx{n}^d\colon \det f_j(\uX)=0 \ \forall j
\right\}.
$$
Note that the \emph{determinantal zeros} $\cZ^{\operatorname{det}}(f_1,\dots,f_\ell)$ are the projection of directional zeros $\cZ^{\operatorname{dir}}(f_1,\dots,f_\ell)$ onto the first component. The motivation behind determinantal zeros comes from noncommutative inequalities. To see this, for a moment assume that $f\in \C\!\Langle \ux\Rangle$ is self-adjoint (i.e., $f$ is fixed by the canonical skew-linear involution on $\C\!\Langle \ux\Rangle$ that fixes the generators). The solution set of a dimension-free matrix inequality
$$\cD(f)=\bigcup_{n\in\N}
\left\{
\uX\in\opm_n(\C)^d\colon X_i=X_i^*,\ f(\uX) \text{ is positive semidefinite}
\right\}$$
is called a \emph{free semialgebraic set} \cite{HKM}; they emerge in control theory, matrix convexity and quantum information theory.
Note that at a boundary point of $\cD(f)$, the value of $f$ is singular (since at least one eigenvalue changes sign); so one can thus view $\cZ^{\operatorname{det}}(f)$ as the ``Zariski closure'' of $\cD(f)$. Often, understanding $\cZ^{\operatorname{det}}(f)$ (a collection of determinantal hypersurfaces) suffices for deducing features of $\cD(f)$.

A simple observation $\cZ^{\operatorname{det}}(f_1f_2)=\cZ^{\operatorname{det}}(f_1)\cup\cZ^{\operatorname{det}}(f_2)$ indicates that determinantal zeros relate to factorization in the free algebra \cite{Coh}. Clearly, every $f\in\px$ can be written as a finite product of irreducible factors in $\px$. However, such a factorization might not be unique; for example, $x_1(x_2x_1+1)=(x_1x_2+1)x_1$. This issue is roughly due to the fact that there are not many invertible elements in $\px$. To ameliorate this, one says that $f,g\in\px$ are \emph{stably associated} \cite[Section 0.5]{Coh} if 
$\left(\begin{smallmatrix}f&0\\0&1\end{smallmatrix}\right)=
P\left(\begin{smallmatrix}g&0\\0&1\end{smallmatrix}\right)Q$ for some $P,Q\in\GL_2(\px)$. For example, $x_1x_2+1$ and $x_2x_1+1$ are stably associated because
$$\begin{pmatrix}x_1x_2+1 & 0 \\ 0 & 1\end{pmatrix}=
\begin{pmatrix}x_1 & 1+x_1x_2 \\ -1 & -x_2\end{pmatrix}
\begin{pmatrix}x_2x_1+1 & 0 \\ 0 & 1\end{pmatrix}
\begin{pmatrix}-x_2 & -1 \\ 1+x_1x_2 & x_1\end{pmatrix}.
$$
Notice that $\cZ^{\operatorname{dir}}(f)=\cZ^{\operatorname{dir}}(g)$ for stably associated $f$ and $g$.
By \cite[Proposition 3.2.9]{Coh}, irreducible factors in a complete factorization of a noncommutative polynomial are unique up to stable associativity. With this notion at hand, one obtains the following correspondence between determinantal zeros and factorization in the free algebra.

\begin{thm}[Helton-Klep-Vol\v{c}i\v{c}]\label{t:HKV1}
$\cZ^{\operatorname{det}}(f_1,\dots,f_\ell)\subseteq\cZ^{\operatorname{det}}(g)$ if and only if for some $j\in\{1,\dots,\ell\}$, every irreducible factor of $f_j$ is stably associated to an irreducible factor of $g$.
\end{thm}

Theorem \ref{t:HKV1} reduces to the case $\ell=1$ after observing that $\cZ^{\operatorname{det}}(f_1,\dots,f_\ell)
\subseteq\cZ^{\operatorname{det}}(g)$ implies $\cZ^{\operatorname{det}}(f_j)
\subseteq\cZ^{\operatorname{det}}(g)$ for some $j$ \cite[Proposition 2.14]{HKV1}. This might seem counter-intuitive at first, and it is due to determinantal zeros being ubiquitous in a certain sense. 
After this reduction, the proof of Theorem \ref{t:HKV1} is given in \cite[Theorem 2.12]{HKV1}, building upon earlier work in \cite{HKV}. The proof differs prominently from those in Sections \ref{sec2} and \ref{sec3}. On one hand, its framework relies on linearization of matrices over the free algebra, and its interplay with factorization. On the other hand, the proof combines results from invariant and representation theory associated with the actions of $\GL_n(\kk)$ on $\mtx{n}^d$ by simultaneous conjugation and the action of $\SL_n(\kk)\times \SL_n(\kk)$ on $\mtx{n}^d$ by simultaneous left-right multiplication, with ampliation techniques from free analysis.

\section{Tracial and weak zeros}\label{sec5}

This section concerns vanishing of linear functionals on images of noncommutative polynomials, which gives it a slight operator-algebraic flavor.
For $f_1,\dots,f_\ell\in\px$ let
$$\cZ^{\operatorname{tr}}(f_1,\dots,f_\ell) = \bigcup_{n\in\N}
\left\{
\uX\in\mtx{n}^d\colon \tr f_j(\uX)=0 \ \forall j
\right\}.
$$
When considering the above \emph{tracial zeros}, one begins by noticing that if $f\in\px$ is a sum of commutators in $\px$, then $\tr f(\uX)=0$ for all matrix tuples $\uX$. The converse is also true \cite[Corollary 4.7]{BK}, which is established using multilinearization and the structure of Lie ideals in $\mtx{n}$. Thus when looking for algebraic certificates of tracial zero inclusions, one should seek for relations between noncommutative polynomials that are satisfied ``up to adding commutators''. In particular, the following Nullstellensatz is given in \cite[Theorem 3.1]{KS}.

\begin{thm}[Klep-\v{S}penko]\label{t:KS}
$\cZ^{\operatorname{tr}}(f_1,\dots,f_\ell)\subseteq\cZ^{\operatorname{tr}}(g)$ if and only if $1$ or $g$ is a linear combination of $f_1,\dots,f_\ell$ and commutators in $\px$.
\end{thm}

While Theorem \ref{t:KS} seemingly consists of two separate cases, $\cZ^{\operatorname{tr}}(f_1,\dots,f_\ell)=\emptyset$ and $\cZ^{\operatorname{tr}}(f_1,\dots,f_\ell)\neq\emptyset$, the proof of \cite[Theorem 3.1]{KS} reduces the latter one to the former one. 
The core arguments of \cite{KS} then utilize bounds on trace identities provided by invariant theory, together with the characterization of trace-zero noncommutative polynomials as sums of commutators \cite{BK}. 

Apart from commutators (indispensable when considering tracial zeros), the other component of Theorem \ref{t:KS} is linear dependence. One can also explore it separately; 
\cite[Theorem 3.7]{BK1} establishes a local-global principle for linear dependence in the free algebra.
Namely, $f_1,\dots,f_\ell\in\px$ are linearly dependent if and only if for all $n\in\N$, $\uX\in\mtx{n}^d$ and $v\in\kk^n$, the vectors $f_1(\uX)v,\dots,f_\ell(\uX)v$ are linearly independent. The proof of \cite[Theorem 3.7]{BK1} 
combines a characterization of linear dependence based on vanishing of Capelli's polynomial,
and a rank constraint on locally linearly dependent operators.
As observed in \cite{Cim}, this local-global linear dependence principle can be interpreted as a Nullstellensatz.
Let
$$\cZ^{\operatorname{w}}(f_1,\dots,f_\ell) = \bigcup_{n\in\N}
\left\{
(\uX,u,v)\in\mtx{n}^d\times\kk^n\times\kk^n\colon u^\ti f_j(\uX)v=0 \ \forall j
\right\}
$$
be the set of \emph{weak zeros} of $f_1,\dots,f_\ell\in\px$ (named in analogy with the weak topology in functional analysis). Then one obtains the following.

\begin{thm}[Bre\v{s}ar-Klep; Cimpri\v{c}]\label{t:BK}
$\cZ^{\operatorname{w}}(f_1,\dots,f_\ell)\subseteq\cZ^{\operatorname{w}}(g)$ if and only if $g$ is a linear combination of $f_1,\dots,f_\ell$.
\end{thm}

The following argument for Theorem \ref{t:BK} is due to Cimpri\v{c}, given in the context of partial differential operators \cite{Cim}. It suffices to assume that $f_1,\dots,f_\ell$ are linearly independent. If $\cZ^{\operatorname{w}}(f_1,\dots,f_\ell)\subseteq\cZ^{\operatorname{w}}(g)$, then for all $n\in\N$, $\uX\in\mtx{n}^d$ and $v\in\kk^n$, the complement of the span of $f_1(\uX)v,\dots,f_\ell(\uX)v$ is contained in the complement of $g(\uX)v$ (here, the complements are considered relative to the standard bilinear form $\kk^n\times\kk^n\to\kk$). Therefore $f_1(\uX)v,\dots,f_\ell(\uX)v,g(\uX)v$ are linearly dependent for every $\uX$ and $v$. By \cite[Theorem 3.7]{BK1}, it follows that $f_1,\dots,f_\ell,g$ are linearly dependent, so $g$ belongs to the span of $f_1,\dots,f_\ell$.

\section{Low-rank values}\label{sec6}

As observed in Section \ref{sec2}, (nonconstant) noncommutative polynomials might not have any matricial zeros, which is an essential barrier to obtaining an exact correspondence between matricial zero sets of noncommutative polynomials and ideals in the free algebra.
For example, $f=1-[x_1,x_2]$ does not have matricial zeros. Nevertheless, $f$ can attain low-rank values on matrix tuples. Concretely, for every $n\in\N$ consider the pair of $n\times n$ matrices
\begin{equation}\label{e:weyl}
X_n=
\begin{pmatrix}
0&1&0&&\\
&0&1&\ddots&\\
&&\ddots&\ddots &0\\
&&&\ddots&1\\
&&&&0
\end{pmatrix},\quad
Y_n=
\begin{pmatrix}
0&&&&\\
1&0&&&\\
0&2&\ddots&&\\
&\ddots&\ddots&\ddots&\\
&&0&n-1&0
\end{pmatrix}.
\end{equation}
A direct calculation shows that $\rk f(X_n,Y_n)=1$ for all $n\in\N$. In particular, while $f$ does not have matricial zeros, the relative rank $\frac1n \rk f(\uX)$ for $\uX\in\mtx{n}^2$ can be arbitrarily small.
The pair of matrices \eqref{e:weyl} is constructed by looking at the canonical representation of the Weyl algebra on $\C[t]$, where $x_1$ acts as multiplication by $t$ and $x_2$ acts as differentiation with respect to $t$, and then truncating compressing these operators to polynomials of degree at most $n-1$.
With such examples (and a previous investigation of noncommutative polynomial zeros in algebraically closed skew fields \cite{ML}) in mind, Makar-Limanov proposed\footnote{
Per author's account, during a talk in Advances in Noncommutative Algebra and Representation Theory at Bar-Ilan University in 2017; but likely also on prior occasions.
} the following statement.

\begin{conj}[Makar-Limanov]\label{c:ML}
If $f\in\px\setminus\kk$ then
$$\inf_{n\in\N} \frac{\min \rk f\left(\mtx{n}^d\right)}{n}=0.$$
\end{conj}

Let us look at another example. The noncommutative polynomial $f=1-[x_1,[x_1,x_2]^2]$ is constantly equal to the identity matrix on all pairs of $1\times 1$ and $2\times 2$ matrices (note that $[X_1,X_2]^2=-\det([X_1,X_2])I$ for $\uX\in\mtx{2}^2$).
On the other hand, for
$$
X_3=
\begin{pmatrix}
 0 & -2 & 0 \\
 \tfrac{1}{6} & 0 & -4 \\
 0 & \tfrac{1}{6} & 0
\end{pmatrix},\quad 
Y_3=
\begin{pmatrix}
 0 & 3 & 0 \\
 \tfrac{1}{4} & 0 & 0 \\
 0 & 0 & \tfrac{3}{2}
\end{pmatrix}
$$
and
$$
X_4=
\begin{pmatrix}
 0 & 1 & 1 & 0 \\
 \tfrac{5}{9} & 0 & -\tfrac{5}{3} & -1 \\
 0 & \tfrac{2}{15} & 0 & -\tfrac{1}{5} \\
 0 & 0 & -\tfrac{5}{3} & 0 
\end{pmatrix},\quad 
Y_4=
\begin{pmatrix}
 0 & -\tfrac{6}{5} & \tfrac{21}{2} & 0 \\
 0 & 0 & -10 & \tfrac{3}{2} \\
 \tfrac{1}{3} & \tfrac{1}{5} & 0 & \tfrac{3}{10} \\
 0 & 0 & \tfrac{5}{2} & 2 
\end{pmatrix}
$$
one has $\rk f(X_3,Y_3)=\rk f(X_4,Y_4)=1$. There is no refined construction behind these two pairs of matrices; they were obtained by solving the system of polynomial equations $f(X_n,Y_n)=nE_{nn}$ (where $E_{nn}$ is a standard matrix unit) using Gr\"obner bases and some ad hoc reductions.
While presenting scarce evidence, this example nevertheless suggests a strengthening of Conjecture \ref{c:ML}: for $f\in\px$, is it true that $\min\rk f(\mtx{n}^d)\le1$ if $f$ is not constant on $\mtx{n}^d$? This question is possibly more amenable to refutation than Conjecture \ref{c:ML}.
Be that as it may, a positive resolution of Conjecture \ref{c:ML} would indicate that \emph{low-rank values} of noncommutative polynomials might be a suitable concept for developing algebraic geometry of noncommutative polynomials on matrices of arbitrary sizes that corresponds to two-sided ideals in $\px$. Namely, one might go further than Conjecture \ref{c:ML} and ask whether the following assertion is true.

\begin{conj}\label{c:true}
For $f_1,\dots,f_\ell,g\in\px$, the following are equivalent:
\begin{enumerate}[(i)]
    \item there is $K\in\N$ such that for all $n\in\N$ and $\uX\in\mtx{n}^d$, 
    $$\rk g(\uX)\le K\cdot\max\{\rk f_1(\uX),\dots,\rk f_\ell(\uX)\};$$
    \item $g\in (f_1,\dots,f_\ell)$.
\end{enumerate}
\end{conj}

The implication (ii)$\Rightarrow$(i) is straightforward: if $g$ is a sum of $K$ terms of the form $pf_jq$ for $p,q\in\px$, then (ii) holds by elementary rank estimates. To approach the converse, it is first necessary to develop a better understanding of low-rank values of noncommutative polynomials, and substantiate their relative abundance.


\end{document}